\documentclass[11pt]{article}

\usepackage{amsfonts,amssymb,amsmath,amsthm,epsfig,euscript}
\usepackage{graphicx}
\usepackage{graphicx,tikz,pgf}
\usepackage{caption}

\newtheorem{theorem}{Theorem}
\newtheorem{lemma}[theorem]{Lemma}

\newtheorem{proposition}[theorem]{Proposition}

\theoremstyle{definition}
\newtheorem{rem}[theorem]{Remark}
\long\def\symbolfootnote[#1]#2{\begingroup
\def\thefootnote{\fnsymbol{footnote}}\footnote[#1]{#2}\endgroup}

\newcommand{\sg}{\sigma}

\newcommand{\cref}[1]{Corollary \ref{corollary:#1}}

\newcommand{\fig}[2]{\begin{figure}[ht]
\centerline{\scalebox{.66}{\epsfig{file=#1.eps}}}
\caption{#2}
\label{fig:#1}
\end{figure}}

\title{Stieltjes moment sequences of polynomials}

\author{
Huyile Liang \\
\small School of Mathematical Sciences\\[-0.8ex]
\small Dalian University of Technology\\[-0.8ex]
\small Dalian 116024, PR China\\[-0.8ex]
\small \texttt{lianghuyile@hotmail.com}
\and
Jeffrey Remmel \\
\small Department of Mathematics\\[-0.8ex]
\small University of California, San Diego\\[-0.8ex]
\small La Jolla, CA 92093-0112. USA\\[-0.8ex]
\small \texttt{remmel@math.ucsd.edu}
\and
Sainan Zheng \\
\small School of Mathematical Sciences\\[-0.8ex]
\small Dalian University of Technology\\[-0.8ex]
\small Dalian 116024, PR China\\[-0.8ex]
\small \texttt{zhengsainandlut@hotmail.com}
}

\date{\small Submitted: Date 1;  Accepted: Date 2;
 Published: Date 3.\\
\small MR Subject Classifications: 05A15, 05E05}

\begin{document}
\maketitle
 This paper is dedicated to the memory of Professor Jeff Remmel, who recently passed away.
\begin{abstract}
\noindent

A sequence $(a_n)_{n \geq 0}$ is Stieltjes moment sequence if it has the form
$a_n = \int_0^\infty x^n d\mu(x)$ for $\mu$ is a nonnegative measure
on $[0,\infty)$. It is known that
$(a_n)_{n \geq 0}$ is a Stieltjes moment sequence if and only if
the matrix $H =[a_{i+j}]_{i,j \geq 0}$ is totally positive, i.e., all
its minors are nonnegative. We define a
sequence of polynomials in $x_1,x_2,\ldots,x_n$ $(a_n(x_1,x_2,\ldots,x_n))_{n \geq 0}$ to be
 a Stieltjes moment sequence of polynomials if
the matrix $H =[a_{i+j}(x_1,x_2,\ldots,x_n)]_{i,j \geq 0}$ is $x_1,x_2,\ldots,x_n$-totally positive, i.e., all
its minors are polynomials in $x_1,x_2,\ldots,x_n$ with nonnegative coefficients.
The main goal of this paper is to
produce a large class of  Stieltjes moment sequences of polynomials
by finding multivariable analogues of Catalan-like numbers as defined by
Aigner.
\end{abstract}

\section{Introduction}
A sequence $(a_n)_{n \geq 0}$ is a {\em Stieltjes moment sequence} if it has
the form
$$a_n = \int_0^\infty x^n d\mu(x)$$ where
$\mu$ is a nonnegative measure on $[0,\infty)$.
There are several other characterizations of  Stieltjes moment sequences.
For example, in \cite[Theorem 1.3]{ST43}, it is proved that a
sequence $(a_n)_{n \geq 0}$ is a Stieltjes moment sequence if and only
if both the matrices $[a_{i+j}]_{0 \leq i,j \leq n}$ and
$[a_{i+j+1}]_{0 \leq i,j \leq n}$ are positive semidefinite for all $n \geq 0$ (see \cite{Min88,Pin10}).

Another characterization Stieltjes moment sequences comes from
the theory of total positivity.
Let $A = [a_{n,k}]_{n,k \geq 0}$
be a finite or infinite matrix. We say that
$A$ is {\em totally positive of order $r$} if all its minors
of order $1,2, \ldots, r$ are nonnegative and we say that
$A$ is {\em totally positive} if it is totally positive of order
$r$ for all $r \geq 1$ (see \cite{Aig01,Ben11,Bre95,Bre96,CLW-EuJC,CLW-LAA} for instance).
Given a sequence $ \alpha = (a_n)_{n \geq 0}$,
we define the {\em Hankel matrix of $\alpha$}, $H(\alpha)$, by
$$H(\alpha) = [a_{i+j}]_{i,j \geq 0} =
\begin{bmatrix} a_0 & a_1 & a_2 & a_3 & \cdots \\
a_1 & a_2 & a_3 & a_4 & \cdots \\
a_2 & a_3 & a_4 & a_5 & \cdots \\
a_3 & a_4 & a_5 & a_6 & \cdots\\
\vdots&\vdots&\vdots&\vdots&\ddots
\end{bmatrix}.
$$
Then it is proved in \cite{Pin10} that $\alpha$ is a
Stieltjes moment sequence if and only if $H(\alpha)$ is TP.

Let $\mathbb{R}$ denote the real numbers and ${\bf x} = x_1,\ldots, x_n$.
In this paper, we may define
when a sequence of polynomials $(a_n({\bf x}))_{n \geq 0}$
in the polynomial ring $\mathbb{R}[{\bf x}]$ is a
Stieltjes moment sequence of polynomials. For any polynomial
$f({\bf x}) = \sum c_{i_1,\ldots,i_n} x_1^{i_1}x_2^{i_2}
\cdots x_n^{i_n}$ in $\mathbb{R}[{\bf x}]$, we let
$f({\bf x})|_{x_1^{i_1}x_2^{i_2}
\cdots x_n^{i_n}} = c_{i_1,\ldots,i_n}$ denote the coefficient of $x_1^{i_1}x_2^{i_2}
\cdots x_n^{i_n}$ in $f({\bf x})$.
We say that $f({\bf x})$ is ${\bf x}$-nonnegative, written
$f({\bf x}) \geq_{{\bf x}} 0$,
if
$$f({\bf x})|_{x_1^{i_1}x_2^{i_2}
\cdots x_n^{i_n}} \geq 0 \ \mbox{for all $i_1,\ldots,i_n$}.$$
Given a pair of polynomials
in $f({\bf x})$ and $g({\bf x})$, we shall write
$$f({\bf x}) \geq _{{\bf x}} g({\bf x})$$
if $f({\bf x}) -g({\bf x}) \geq_{{\bf x}} 0$.
Let $M = [m_{n,k}({\bf x})]_{n,k \geq 0}$
be a finite or infinite matrix of polynomials in
$\mathbb{R}[{\bf x}]$. We say that
$M$ is {\em ${\bf x}$-totally positive of order $r$} (${\bf x}$-$TP_r$) if all its minors
of order $1,2, \ldots, r$ are polynomials in ${\bf x}$ with nonnegative coefficients and we say that
$M$ is {\em ${\bf x}$-totally positive} (${\bf x}$-$TP$) if it is
${\bf x}$-totally positive of order
$r$ for all $r \geq 1$.

Given a sequence $ \alpha = (a_k({\bf x}))_{k \geq 0}$
of polynomials in $\mathbb{R}[{\bf x}]$,
we define the {\em Hankel matrix of $\alpha$}, $H(\alpha,{\bf x})$, by the following
$$H(\alpha,{\bf x}) = [a_{i+j}({\bf x})]_{i,j \geq 0}
=\begin{bmatrix} a_0({\bf x}) & a_1({\bf x}) & a_2({\bf x}) & a_3({\bf x}) & \cdots \\
a_1({\bf x}) & a_2({\bf x}) & a_3({\bf x}) & a_4({\bf x}) & \cdots \\
a_2({\bf x}) & a_3({\bf x}) & a_4({\bf x}) & a_5({\bf x}) & \cdots \\
a_3({\bf x}) & a_4({\bf x}) & a_5({\bf x}) & a_6({\bf x}) & \cdots\\
\vdots&\vdots&\vdots&\vdots&\ddots
\end{bmatrix}.
$$
Then if $\alpha$ is a
{\em Stieltjes moment sequence of polynomials} if and only if $H(\alpha,{\bf x})$
is ${\bf x}$-$TP$.  In the case where $n=1$ so that
we are considering polynomials
in a single variable,
our definition coincides with the definition of Stieltjes moment
sequences of polynomials as defined by Wang and Zhu \cite{WZ16}.

The main goal of this paper is to produce a number
of combinatorially defined Stieltjes moment sequences of polynomials.  We
shall do this by finding appropriate multivariable
analogues of Catalan-like numbers
as defined by Aigner \cite{Aig99}. Aigner's idea is the following.
Let $\sg= (s_k)_{k \geq 0}$ and $\tau = (t_{k+1})_{k \geq 0}$ be two
sequences of nonnegative numbers.  Then define an
infinite lower triangular matrix $A:=A^{\sg,\tau} =[a_{n,k}]_{n,k \geq 0}$
where the $a_{n,k}$s are defined by the recursions
\begin{equation}\label{arec}
a_{n+1,k} = a_{n,k-1}+s_ka_{n,k} +t_{k+1}a_{n,k+1}
\end{equation}
subject to the initial conditions that $a_{0,0} = 1$ and
$a_{n,k} =0$ unless $n \geq k \geq 0$. Aigner called
$A^{\sg,\tau}$ the {\em recursive matrix} corresponding
to $(\sg,\tau)$ and he called the
sequence $(a_{n,0})_{n \geq 0}$, the {\em Catalan-like numbers} corresponding
to $(\sg,\tau)$. Recently, Liang et al. \cite{LMW-DM} showed that
many Catalan-like numbers are Stieltjes moment sequences
by proving that the Hankel matrix of the sequence $(a_{n,0})_{n \geq 0}$
is totally positive.
Such examples include the Catalan numbers, the Bell numbers, the
central Delannoy numbers, the restricted hexagonal numbers, the central
binomial coefficients, and the large Schr\"oder numbers.

Liu and Wang \cite{LW2007} defined
 a sequence of polynomials $(f_n(q))_{n \geq 0}$ over $\mathbb{R}$ to be
{\em $q$-log convex} ($q$-$LCX$) if for all $n \geq 1$,
\begin{equation}
(f_n(q))^2 \geq _q f_{n-1}(q)f_{n+1}(q)
\end{equation}
and defined a sequence of polynomials
$(f_n(q))_{n \geq 0}$ to be
{\em strongly $q$-log convex} ($q$-$SLCX$) if for all $n \geq m \geq 1$,
\begin{equation}
f_n(q)f_m(q) \geq _q f_{n-1}(q)f_{m+1}(q).
\end{equation}
Zhu \cite{Zhu13} produced many examples of $q$-$SLCX$ sequences
of polynomials by modifying Aigner's Catalan-like numbers.
In such a situation, Zhu \cite{Zhu13} showed that the
sequence of polynomials $(m_{n,0}(q))_{n \geq 0}$ is
a $q$-$SLCX$ sequence of polynomials if for all
$k \geq 0$, $s_k(q)s_{k+1}(q) -t_{k+1}(q)r_{k+1}(q) \geq_q 0$.
However, it is not the case that
such a sequence of polynomials $(m_{n,0}(q))_{n \geq 0}$ is always
a Stieltjes moment sequence of polynomials. For example,
suppose that $a$ and $b$ are nonnegative real numbers and
$r_k(q) =1$ for $k \geq 1$, $s_0(q) =q^2$ and
$s_k(q) =1+q^2 +a*q^b$ for
$k \geq 1$, and $t_1(q) =q^4$ and $t_k(q) =q^2+q^4$ for $k \geq 2$.
It is easy to check that for all $k \geq 0$, $s_k(q)s_{k+1}(q) -t_{k+1}(q)r_{k+1}(q) \geq_q 0$. First one can compute that
\begin{eqnarray*}
m_{0,0}(q) &=& 1, \\
m_{1,0}(q) &=& q^2, \\
m_{2,0}(q) &=& q^4+4q^6+aq^{4+b}, \\
m_{3,0}(q) &=& q^4+5q^6+9q^8+2aq^{4+b}+4aq^{6+b}+a^2q^{4+2b}, \ \mbox{and}\\
m_{4,0}(q) &=& q^4+8q^6+20q^8+21q^{10}+3aq^{4+b} +13aq^{6+b} + \\
&& 15aq^{8+b}+3a^2q^{4+2b}+5a^2q^{6+2b}+a^3q^{4+3b}.
\end{eqnarray*}
Then one can compute that
\begin{multline*}
\mathrm{det}\left(\begin{bmatrix} m_{0,0}(q) & m_{1,0}(q) & m_{2,0}(q) \\
m_{1,0}(q) & m_{2,0}(q) & m_{3,0}(q) \\
m_{2,0}(q) & m_{3,0}(q) & m_{4,0}(q)
\end{bmatrix}\right) = \\
-q^8-4 q^{10}+6 q^{12}+36 q^{14}+27 q^{16}-64 q^{18}-3 a q^{8+b}-2 a q^{10+b}+27 a q^{12+b}+35 a q^{14+b}-\\
48 a q^{16+b}-3 a^2 q^{8+2 b}+5
a^2 q^{10+2 b}+14 a^2 q^{12+2 b}-12 a^2 q^{14+2 b}-a^3 q^{8+3 b}+3 a^3 q^{10+3 b}-a^3 q^{12+3 b}
\end{multline*}
which is not a polynomial in $q$ with nonnegative coefficients for
all integers $a,b \geq 0$.

Wang and Zhu \cite{WZ16} showed that many of the special sequences considered
by Zhu \cite{Zhu13} are in fact Stieltjes moment sequences of polynomials over $q$.
These include the following well-known polynomials which are $q$-analogues
of Catalan-like numbers.
\begin{enumerate}
\item The Bell polynomials $B_n(q) = \sum_{k=0}^n S(n,k)q^k$ when
$r_k(q) =1$, $s_k(q) =k+q$, and $t_k(q) =kq$. Here $S(n,k)$ is the Stirling
number of the second kind which counts the number of set partitions
of $\{1, \ldots, n\}$ into $k$ parts.
\item The Eulerian polynomials $A_n(q) = \sum_{k=0}^{n} A(n,k)q^k$ when
$r_k(q) =1$, $s_k(q) =(k+1)q+k$, and $t_k(q) =k^2q$.
Here $A(n,k)$ is the number of
permutations of $n$ with $k$ descents.
\item The $q$-Schr\"oder numbers,  $r_n(q) = \sum_{k=0}^n \frac{1}{k+1}
\binom{2k}{k} \binom{n+k}{n-k}q^k$ when $r_k(q) =1$, $s_0(q) =1+q$,
$s_k(q) =1+2q$ for $k \geq 1$, and $t_k(q) =q(1+q)$.
\item The $q$-central Delannoy numbers
$D_n(q) = \sum_{k=0}^n \binom{n+k}{n-k} \binom{2k}{k} q^k$ when
$r_k(q) =1$, $s_k(q) =1+2q$, $t_1(q) =2q(q+1)$,
 and $t_k(q) =q(1+q)$ for $k >1$.
\item The Narayana polynomials
$N_n(q) = \sum_{k=1}^n \frac{1}{n} \binom{n}{k} \binom{n}{k-1} q^k$ when
$r_k(q) =1$, $s_0(q) =q$, $s_k(q) =1+q$ for $n \geq 1$,  and $t_k(q) =q$.

\item The Narayana polynomials $W_n(q) = \sum_{k=0}^n \binom{n}{k}^2 q^k$ of type $B$
when $r_k(q) =1$, $s_k(q) =1+q$, $t_1(q) = 2q$,  and $t_k(q) =q$ for $k > 1$.
\end{enumerate}

In this paper, we consider multivariable analogues
Aigner's Catalan-like numbers. That is,
suppose that we are given three sequences of polynomials over
$\mathbb{R}$ with nonnegative coefficients
$$\pi = (r_k({\bf x}))_{k \geq 1},\  \sg = (s_k({\bf x}))_{k \geq 0}, \ \mbox{and} \
\tau = (t_{k+1}({\bf x}))_{k \geq 0}.$$
Then we define a lower triangular matrix of polynomials
$$M({\bf x}):=M^{\pi,\sg,\tau}({\bf x}) = [m_{n,k}({\bf x})]_{0 \leq k \leq n}$$
where the $m_{n,k}({\bf x})$ are defined by the recursions
\begin{equation}\label{rec:main}
m_{n+1,k}({\bf x}) =  r_k({\bf x}) m_{n,k-1}({\bf x}) + s_k({\bf x})m_{n,k}({\bf x}) + t_{k+1}({\bf x})m_{n,k+1}({\bf x})
\end{equation}
subject to the initial conditions that $m_{0,0}({\bf x}) =1$ and
$m_{n,k}({\bf x}) = 0$ unless $0 \leq k \leq n$.

We note that one can give simple combinatorial interpretations
of the polynomial $m_{n,k}({\bf x})$ defined by the
recursions (\ref{rec:main})
in terms of Motzkin paths.
A Motzkin path
is path that starts at $(0,0)$ and consist of three types of
steps, up-steps $(1,1)$, down-steps $(1,-1)$, and level-steps
$(1,0)$.  We let $\mathcal{M}_{n,k}$ denote the set all
paths that start at $(0,0)$, end at $(n,k)$, and  stays on or above the
$x$-axis. We weight an up-step at that ends at level $k$ with
$r_k({\bf x})$, a level-step that ends at level $k$ with $s_k({\bf x})$, and
a down-step that ends at level $k$ with $t_{k+1}({\bf x})$. See
Figure \ref{fig:steps}. Given a path $P$
in $\mathcal{M}_{n,k}$, we let the weight of $P$, $w(P)$, equal
the product of all the weights of the steps in $P$. Then if
we let
$$m_{n,k}({\bf x}) = \sum_{P \in \mathcal{M}_{n,k}}w(P),$$
it is easy to see that the $m_{n,k}({\bf x})$ satisfy the recursions
(\ref{rec:main}).

\fig{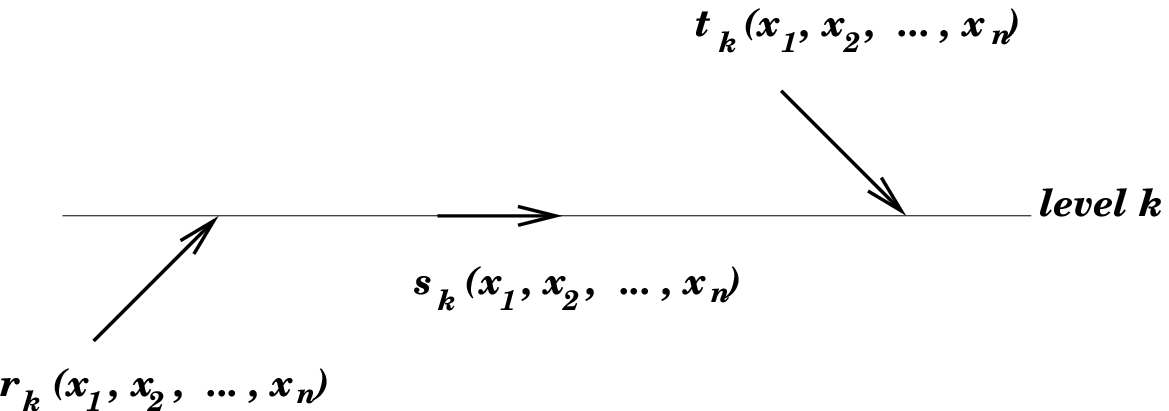}{The weight of steps in Motzkin paths}

Suppose that we are given three sequence of polynomials over
$\mathbb{R}$ with nonnegative coefficients
$\pi = (r_k({\bf x}))_{k \geq 1}$, $\sg = (s_k({\bf x}))_{k \geq 0}$, and
$\tau = (t_{k+1}({\bf x}))_{k \geq 0}$.
The main goal of this paper is to give some necessary conditions
that will ensure that the sequence of polynomials
$(m_{s,0}({\bf x}))_{s \geq 0}$ is a Stieltjes sequence of polynomials
where the polynomials $m_{n,k}({\bf x})$ are defined by (\ref{rec:main}). First we will
prove that $(m_{s,0}({\bf x}))_{s \geq 0}$  a Stieltjes sequence of polynomials
if the matrix
$$
J({\bf x}):=J^{(\pi,\sigma,\tau)}({\bf x}) =
\begin{bmatrix} s_0({\bf x}) & r_1({\bf x}) &  &  &  \\
t_1({\bf x}) & s_1({\bf x}) & r_2({\bf x}) &  &  \\
& t_2({\bf x}) & s_2({\bf x}) & r_3({\bf x}) &  \\
 &&\ddots &\ddots & \ddots   \\
\end{bmatrix}
$$
is ${\bf x}$-TP and $r_k({\bf x}) =1$ for all $k \geq 0$.
Following
Wang and Zhu \cite{WZ16}, we shall show that if there are
sequences of polynomials with nonnegative coefficients
in $\mathbb{R}[{\bf x}]$,
\begin{equation*}
(b_1({\bf x}),b_2({\bf x}), \ldots) \ \mbox{and} \
(c_1({\bf x}),c_2({\bf x}), \ldots)
\end{equation*}
 such that
\begin{eqnarray*}
s_{n}({\bf x})&=& b_{n+1}({\bf x})+c_{n+1}({\bf x}) \
\mbox{for $n \geq 0$}, \\
t_n({\bf x}) &=& c_n({\bf x})b_{n+1}(\bf x) \ \mbox{for $n \geq 0$, and} \\
r_n({\bf x}) &=&1 \ \mbox{for $n \geq 0$,}
\end{eqnarray*}
then $J^{(\pi,\sg,\tau)}({\bf x})$ is ${\bf x}$-TP.\\
We shall also show that
$J^{(\pi,\sg,\tau)}({\bf x})$ is
${\bf x}$-TP if the following conditions hold:
\begin{enumerate}
\item [\rm (i)]$s_0({\bf x})-1\geq_{\bf x} 0$,
\item [\rm (ii)] $s_{i}({\bf x})s_{i+1}({\bf x}) - t_{i+1}({\bf x})r_{i+1}({\bf x})\geq_{\bf x} 0$ for all $i \geq 0$,
\item [\rm (iii)] $s_{i+1}({\bf x}) - t_{i+1}({\bf x})r_{i+1}({\bf x}) -1 \geq_{\bf x} 0$ for all $i \geq 0$.
\end{enumerate}
We will then use these facts to produce many examples of
Stieltjes sequences of polynomials.

We should note that one of the advantages of producing
Stieltjes moment sequences of polynomials $(a_k({\bf x}))_{k \geq 0}$
is that we automatically obtain infinitely many other examples of the
Stieltjes moment sequences of polynomials by replacing
$x_i$ by $\phi_i(\bf x)$ where $\phi_i({\bf x})$ is a polynomial with nonnegative coefficients.  That is,
$$(a_k(\phi_1({\bf x}),
\ldots, \phi_n({\bf x})))_{k \geq 0}$$
 will be another Stieltjes moment
sequence of polynomials.  In addition, we can produce
infinitely many example of Stieltjes moment sequences by
replacing $x_i$ by any nonnegative real  number $r_i$.  That is, if
$r_i \geq 0$ for $i=1, \ldots, n$, then $(a_k(r_1,\ldots,r_n))_{k \geq 0}$ is a Stieltjes moment sequence.

The outline of this paper is as follows. In Section 2,
we shall give several sufficient conditions which ensure
that the $H(\alpha,{\bf x})$ is ${\bf x}$-totally
positive for a sequence of polynomials
$\alpha =(a_k({\bf x}))_{k \geq 0}.$ Then in Section 3,
we shall uses these results to produce many combinatorial sequences
defined Stieltjes moment sequences of polynomials $(a_k({\bf x}))_{k \geq 0}$.

\section{Premliminaries}
We start with two lemmas about $n \times n$ tridiagonal matrices of
polynomials in ${\bf x}$. Suppose that $J:=J({\bf x}) =[a_{i,j}({\bf x})]_{i,j =1, \ldots ,n}$ is
tridiagonal matrix of nonnegative polynomials in ${\bf x}$ over $\mathbb{R}$.
That is, $a_{i,j}({\bf x}) \geq_{\bf x} 0$ for all $i,j$ and $a_{i,j}({\bf x}) = 0$ if
$|i-j| > 1$. Let $J[\{i_1, \ldots, i_k\},\{j_1, \ldots, j_k\}]$
denote the $k \times k$ matrix which arises  from $J$ by taking
the elements that lie in the intersection of the rows
$i_1, \ldots, i_k$ in $J$ and the columns $j_1, \ldots, j_k$ in $J$ where
$1 \leq i_1 < \cdots < i_k \leq n$ and $1 \leq j_1 < \cdots < j_k \leq n$.
We say that a minor
$\mathrm{det}J[\{i_1, \ldots, i_k\},\{j_1, \ldots, j_k\}]$ is a consecutive
principal minor if there exists an $s$ such that
$(i_1,i_2, \ldots, i_k) = (j_1,j_2, \ldots,j_k) = (s,s+1,\ldots,s+k-1)$.

\begin{lemma}\label{basic}
Suppose that $J=[a_{i,j}({\bf x})]_{i,j =1, \ldots ,n}$ is
tridiagonal matrix of nonnegative polynomials in $\mathbb{R}[{\bf x}]$.
Then $J$ is ${\bf x}$-TP if and only if all of its consecutive
principle minors are polynomials in ${\bf x}$ with nonnegative coefficients.
\end{lemma}
\begin{proof}

Consider a minor
$\mathrm{det}J[\{i_1, \ldots, i_k\},\{j_1, \ldots, j_k\}]$.
First we observe that \\
$\mathrm{det}J[\{i_1, \ldots, i_k\},\{j_1, \ldots, j_k\}] =0$ if
there is an $s$ such that $|i_s -j_s|> 1$.  That is, suppose
that  $i_s < j_s$. Clearly
$a_{i_s,j_t}({\bf x}) =0$ for $t = s, s+1, \ldots, k$. But then
 $$\mathrm{det}J[\{i_1, \ldots, i_k\},\{j_1, \ldots, j_k\}] =
\sum_{\sg \in S_k} \mathrm{sgn}(\sg) a_{i_1,j_{\sg_1}}({\bf x}) \cdots
a_{i_k,j_{\sg_k}}({\bf x})$$
where $S_k$ is the symmetric group.
It follows that if $a_{i_s,j_{\sg_s}}({\bf x}) \neq 0$, then $\sg_s \leq s-1$. But in such a situation,
there must be an $r <s$ such that $\sg_r \geq s$ which would imply
that $a_{i_r,j_{\sg_r}}({\bf x}) =0$. Thus
$\mathrm{det}J[\{i_1, \ldots, i_k\},\{j_1, \ldots, j_k\}] =0$.
A similar argument can be used to show that if
$j_s < i_s$, then
 $\mathrm{det}J[\{i_1, \ldots, i_k\},\{j_1, \ldots, j_k\}] =0$.

Thus we only have to consider minors of the form
$\mathrm{det}J[\{i_1, \ldots, i_k\},\{j_1, \ldots, j_k\}]$ where
$|i_s-j_s|\leq 1$. In such a situation, let $1\leq t_1 < \cdots < t_r \leq k$
be the indices $t$ such that such that $|i_t -j_t| =1$. Then
it is easy to see that
$\mathrm{det}J[\{i_1, \ldots, i_k\},\{j_1, \ldots, j_k\}]$ equals
$\prod_{s = 1}^r a_{i_{t_s},j_{t_s}}({\bf x})$ times a product of
consecutive principle minors involving consecutive indices. For example,
\begin{multline*}\mathrm{det}J[\{1,2,3,5,6,9,11,12\},\{2,3,4,5,6,8,11,12\}] = \\
a_{1,2}({\bf x})a_{2,3}({\bf x})a_{3,4}({\bf x})\mathrm{det}J[\{5,6\},\{5,6\}]a_{9,8}({\bf x})
\mathrm{det}J[\{11,12\},\{11,12\}].
\end{multline*}
Thus it follows that if all the consecutive principal minors
are polynomials in ${\bf x}$ with nonnegative coefficients, then
$J$ is  ${\bf x}$-TP.
\end{proof}

\begin{lemma}\label{key}
Suppose that
$$
J({\bf x}) =
\begin{bmatrix} s_0({\bf x}) & r_1({\bf x}) &  &  &  \\
t_1({\bf x}) & s_1({\bf x}) & r_2({\bf x}) &  &  \\
& t_2({\bf x}) & s_2({\bf x}) & r_3({\bf x}) &  \\
 &&\ddots &\ddots & \ddots   \\
\end{bmatrix}
$$
is
tridiagonal matrix of nonnegative polynomials in $\mathbb{R}[{\bf x}]$,
where $\sg=(s_i({\bf x}))_{i \geq 1}$, $\pi =(r_i({\bf x}))_{i \geq 0}$, and $
\tau =(t_{i+1}({\bf x}))_{i \geq 0}$ are sequences of
non-zero polynomials over $\mathbb{R}$ with nonnegative coefficients such
that
\begin{enumerate}
\item [\rm (i)]$s_0({\bf x})-1\geq_{\bf x} 0$,
\item [\rm (ii)] $s_{i}({\bf x})s_{i+1}({\bf x}) - t_{i+1}({\bf x})r_{i+1}({\bf x})\geq_{\bf x} 0$ for all $i \geq 0$,
\item [\rm (iii)] $s_{i+1}({\bf x}) - t_{i+1}({\bf x})r_{i+1}({\bf x}) -1 \geq_{\bf x} 0$ for all $i \geq 0$.
\end{enumerate}
Then $J({\bf x})$ is ${\bf x}$-TP.
\end{lemma}
\begin{proof}
By Lemma \ref{basic}, we need only show that all the consecutive
principal minors of $J({\bf x})$ are nonnegative polynomials in ${\bf x}$.

Let $M_k({\bf x}) = \mathrm{det}J[\{1,\ldots,k\},\{1,\ldots, k\}]$.
First we shall prove by induction that $M_k({\bf x}) \geq_{\bf x} 0$ for all
$1 \leq k \leq n$ and $M_k({\bf x})-M_{k-1}({\bf x}) \geq_{\bf x} 0$ for
all $2 \leq k \leq n$.
Note that $M_1({\bf x}) = s_0({\bf x}) \geq_{\bf x} 0$ and
$ M_2({\bf x})= s_0({\bf x})s_1({\bf x}) -t_1({\bf x})r_1({\bf x}) \geq_{\bf x} 0$ by assumption.
Then note that
\begin{eqnarray*}
M_2({\bf x})-M_1({\bf x}) &=& s_0({\bf x})s_1({\bf x}) -t_1({\bf x})r_1({\bf x}) -s_0({\bf x})\\
&=&s_0({\bf x})(s_1({\bf x}) -t_1({\bf x})r_1({\bf x})-1)+t_1({\bf x})r_1({\bf x})(s_0({\bf x})-1) \geq_{\bf x} 0
\end{eqnarray*}
since we are assuming that that $s_0({\bf x}) -1 \geq_{\bf x} 0$ and $s_1({\bf x}) -t_1({\bf x})r_1({\bf x})-1 \geq_{\bf x} 0$.

Next assume $k \geq 3$ and by induction that
$M_{k-1}({\bf x}) \geq_{\bf x} 0$ and $M_{k-1}({\bf x}) -M_{k-2}({\bf x}) \geq_{\bf x} 0$.
Then expanding $M_k$ about the last row, we see
that
\begin{eqnarray*}
M_k({\bf x}) &=& s_k({\bf x})M_{k-1}({\bf x}) - r_k({\bf x})t_k({\bf x})M_{k-2}({\bf x}) \\
&=&
(s_k({\bf x})-r_k({\bf x})t_k({\bf x}))M_{k-1}({\bf x}) +r_k({\bf x})t_k({\bf x})(M_{k-1}({\bf x}) -M_{k-2}({\bf x}))\geq_{\bf x} 0.
\end{eqnarray*}
Similarly,
$$M_k({\bf x}) - M_{k-1}({\bf x}) =
(s_k({\bf x})-r_k({\bf x})t_k({\bf x})-1)M_{k-1}({\bf x}) +r_k({\bf x})t_k({\bf x})(M_{k-1}({\bf x}) -M_{k-2}({\bf x}))\geq_{\bf x} 0.$$
Thus it follows that $M_k({\bf x}) \geq_{\bf x} 0$ for all $k$.

For consecutive principle minors of the form
$\mathrm{det}J[\{m,m+1, \ldots, m+r\},\{m,m+1, \ldots, m+r\}]$
where $m > 1$, we note $s_m({\bf x}) -1 \geq _{\bf x} 0$ since
$s_m({\bf x}) -t_m({\bf x})r_m({\bf x}) -1 \geq _{\bf x} 0$.  It
then follows that the matrix $J[\{s,s+1, \ldots, s+r\},\{s,s+1, \ldots, s+r\}]$ satisfies the hypothesis of the theorem
so that the proof that $M_n({\bf x}) \geq_{\bf x} 0$ for all $n \geq 1$ can
be applied to show that
$$\mathrm{det}J[\{s,s+1, \ldots, s+r\},\{s,s+1, \ldots, s+r\}]
\geq_{\bf x} 0.$$
\end{proof}
Wang and Zhu \cite[Lemma 3.3]{WZ16},
showed that if $b_n(q)$ and $c_n(q)$ be all $q$-nonnegative,
then the corresponding tridiagonal matrix is $q$-TP.
The method of the proof used in \cite[Lemma 3.3]{WZ16} can be carried over verbatim to its ${\bf x}$-analogue.
Here we omit the details for brevity.
\begin{lemma}\label{product}
Let $(b_1({\bf x}),b_2({\bf x}), \ldots)$ and
$(c_1({\bf x}),c_2({\bf x}), \ldots)$
be sequences of polynomials in  $\mathrm{R}[{\bf x}]$
with nonnegative
coefficients. Then the tridiagonal matrix
$$J^{b,c} = \begin{bmatrix}
b_1({\bf x}) + c_1({\bf x}) & 1 & & \\
b_2({\bf x})c_1({\bf x})& b_2({\bf x}) + c_2({\bf x})&1 & \\
& b_3({\bf x})c_2({\bf x})& b_3({\bf x}) + c_3({\bf x})&  \ddots \\
& \ddots &\ddots
\end{bmatrix} $$
is ${\bf x}$-TP.
\end{lemma}

\begin{theorem}\label{HMJ}
 Let
$J =J^{(\pi,\sg,\tau)}({\bf x})$ be the tridiagonal matrix
$$J  =
\begin{bmatrix} s_0({\bf x}) & r_1({\bf x}) &  &  &   \\
t_1({\bf x}) & s_1({\bf x}) & r_2({\bf x}) &  &  \\
& t_2({\bf x}) & s_2({\bf x})  & r_3({\bf x}) &  \\
 &\ddots &\ddots & \ddots &  \\
 && t_{n-1}({\bf x}) & s_{n-1}({\bf x}) &r_{n}({\bf x}) \\
&&\ddots&\ddots&\ddots \end{bmatrix}$$
where $\sg=(s_i({\bf x}))_{i \geq 1}$, and $
\tau =(t_{i+1}({\bf x}))_{i \geq 0}$ are sequences of
non-zero polynomials over $\mathbb{R}$ with non-negative coefficients
and $r_k({\bf x}) =1$ for all $k \geq 0$.
Let $M({\bf x})$ be the a lower triangular matrix of polynomials
$$M({\bf x}):=M^{\pi,\sg,\tau}({\bf x}) = [m_{n,k}({\bf x})]_{0 \leq k \leq n}$$
where the $m_{n,k}({\bf x})$ are defined by (\ref{rec:main}).
Then if the coefficient matrix $J$ is ${\bf x}$-totally positive,
the sequence $(m_{n,0}({\bf x}))_{n \geq 0}$ is a Stieltjes moment sequence of polynomials.
\end{theorem}
\begin{proof}
Let $M({\bf x})=[m_{i,j}({\bf x}) ]_{i,j \geq 0}$ and
$H({\bf x})=[m_{i+j,0}({\bf x})]_{i,j\ge 0}$  be the Hankel matrix of the ${\bf x}$-Catalan-like numbers $m_{n,0}({\bf x})$.
We need to show that $H({\bf x})$ is ${\bf x}$-totally positive. Let $T_0({\bf x})=1, T_k({\bf x})=t_1({\bf x})\cdots t_k({\bf x})$ and $T=\textrm{diag} (T_0({\bf x}),T_1({\bf x}),T_2({\bf x}),\ldots)$.
Then it is not difficult to verify that
$$H({\bf x})=M({\bf x})T({\bf x})M({\bf x})^t,$$ see \cite[(2.5)]{Aig01}.
By the Cauchy-Binet Theorem, any minor of $H({\bf x})$ can be expressed of
sums of products of minors of $M({\bf x})$, $T({\bf x})$, and $M({\bf x})^t$.
It follows that if $M({\bf x})$, $T({\bf x})$, and $M({\bf x})^t$ are ${\bf x}$-TP matrices,
then $H({\bf x})$ is ${\bf x}$-TP.
It is clear that $T({\bf x})$ is ${\bf x}$-TP and
$M({\bf x})$ is ${\bf x}$-TP if and only  if $M({\bf x})^t$ is ${\bf x}$-TP. Thus
to show that $H({\bf x})$ is ${\bf x}$-TP, we need only show that $M({\bf x})$ is
${\bf x}$-TP.
Thus we need only show that $J(x)$ being ${\bf x}$-TP
implies $M({\bf x})$ is ${\bf x}$-TP.
Let $M_n({\bf x})=[m_{i,j}(\bf x)]_{0\le i,j\le n}$ be the $n$-th leading principal submatrix of $M({\bf x})$.
Clearly, to show that $M({\bf x})$ is ${\bf x}$-TP, it suffices to show that $M_n({\bf x})$ are ${\bf x}$-TP for $n\ge 0$.
We do this by induction on $n$.
Obviously, $M_0({\bf x})$ is ${\bf x}$-TP.
Assume that $M_n({\bf x})$ is ${\bf x}$-TP.
Then one can easily show that (\ref{rec:main}) implies that
\begin{equation*}\label{r=rj}
M_{n+1}({\bf x})=\left[\begin{array}{cc}e_{n+1}({\bf x}) \\ 0_n & M_n({\bf x})\\\end{array}\right]
\left[\begin{array}{c}e_{n+1}({\bf x})\\ J_n({\bf x})\\\end{array}\right],
\end{equation*}
where $e_{n+1}({\bf x})=[1,0,\ldots,0]$, $0_n$ is column of $n$ 0s, and $J_n({\bf x})$ is the
$n \times ( n+1)$ principal submatrix of $J({\bf x})$.
By the induction hypothesis, $M_n({\bf x})$ is ${\bf x}$-TP so that
$\left[\begin{array}{cc}1 & 0\\ 0 & M_n({\bf x})\\\end{array}\right]$ is ${\bf x}$-TP.
On the other hand,
$J_n({\bf x})$ is ${\bf x}$-TP since it is a submatrix of the
${\bf x}$-TP matrix $J({\bf x})$,
so that $\left[\begin{array}{c}e_{n+1}({\bf x})\\ J_n({\bf x})\\\end{array}\right]$ is ${\bf x}$-TP. Applying the Cauchy-Binet Theorem again, we see that
$M_{n+1}({\bf x})$ is ${\bf x}$-TP.
\end{proof}

Given a polynomial $a({\bf x}) = \sum_{(i_1, \ldots, i_n) \in I} c_{i_1, \ldots, i_n} x_1^{i_1} \cdots x_n^{i_n}$ where $I$ is finite index set and $c_{i_1, \ldots, i_n} \neq 0$ for
all $(i_1, \ldots, i_n) \in I$, we let the degree of $a({\bf x})$, $\mathrm{deg}(a({\bf x}))$, equal $\max(\{i_1+ \cdots + i_n: (i_1, \ldots, i_n) \in I\})$.  We say that $a({\bf x})$ is homogeneous of degree $n$ if
$i_1+ \cdots + i_n =n$ for all $(i_1, \ldots, i_n) \in I$ and is inhomogeneous otherwise.  If  $a({\bf x})$ had degree $n$, then we let
$$H_{x_0}(a({\bf x})) =x_0^n a\left(\frac{x_1}{x_0},\frac{x_2}{x_0}, \ldots ,\frac{x_n}{x_0}\right).$$
For example, if $a(x_1,x_2) = 1+x_1 +x_1 x_2 +x_1^3$, then
$$H_{x_0}(a(x_1,x_2))= x_0^3\left(1+\frac{x_1}{x_0} +\frac{x_1}{x_0}\frac{x_2}{x_0}+\frac{x_1}{x_0}\frac{x_1}{x_0}\frac{x_1}{x_0}\right) = x_0^3 + x_0^2 x_1 + x_0x_1x_2+x_1^3.$$
Clearly if $\mathrm{deg}(a({\bf x}))=n$, then
$H_{x_0}(a({\bf x}))$ is a homogeneous polynomial of degree $n$.
\begin{theorem}\label{H}  Suppose that $\alpha = (a_0({\bf x}),a_1({\bf x}),a_2({\bf x}),\ldots )$ is a Stieltjes moment sequence of polynomials such that for all $n \geq 0$,
$\mathrm{deg}(a_n({\bf x}))=n$.
Then $H_{x_0}(\alpha) = (H_{x_0}(a_0({\bf x})), H_{x_0}(a_1({\bf x})),\ldots )$ is a Stieltjes moment sequence of polynomials.
\end{theorem}
\begin{proof}
Let
$${\bf H} = [H_{x_0}(a_{i+j}({\bf x}))]_{i,j \geq 0} =
\begin{bmatrix} H_{x_0}(a_0({\bf x})) & H_{x_0}(a_1({\bf x})) & H_{x_0}(a_2({\bf x})) & H_{x_0}(a_3({\bf x})) & \ldots \\
H_{x_0}(a_1({\bf x})) & H_{x_0}(a_2({\bf x})) & H_{x_0}(a_3({\bf x})) & H_{x_0}(a_4({\bf x})) & \ldots \\
H_{x_0}(a_2({\bf x})) & H_{x_0}(a_3({\bf x})) & H_{x_0}(a_4({\bf x})) & H_{x_0}(a_5({\bf x})) & \ldots \\
H_{x_0}(a_3({\bf x})) & H_{x_0}(a_4({\bf x})) & H_{x_0}(a_5({\bf x})) & H_{x_0}(a_6({\bf x})) & \ldots\\
\vdots&\vdots&\vdots&\vdots&\ddots
\end{bmatrix}
$$
and
$$
H= H(\alpha,{\bf x}) = [a_{i+j}({\bf x})]_{i,j \geq 0} =
\begin{bmatrix} a_0({\bf x}) & a_1({\bf x}) & a_2({\bf x}) & a_3({\bf x}) & \ldots \\
a_1({\bf x}) & a_2({\bf x}) & a_3({\bf x}) & a_4({\bf x}) & \ldots \\
a_2({\bf x}) & a_3({\bf x}) & a_4({\bf x}) & a_5({\bf x}) & \ldots \\
a_3({\bf x}) & a_4({\bf x}) & a_5({\bf x}) & a_6({\bf x}) & \ldots\\
\vdots&\vdots&\vdots&\vdots&\ddots
\end{bmatrix}.
$$

Since $\alpha$ is a Stieltjes moment sequence of polynomials, we know that if $1 \leq i_1 < \cdots < i_k$ and $1 \leq j_1 < \cdots < j_k$, then
the minor $\mathrm{det}(H[\{i_1,\ldots, i_k\},\{j_1,\ldots, j_k\}])$ equals
$$\sum_{\sg \in S_k} \mathrm{sgn}(\sg) \prod_{s=1}^k a_{i_s-1+j_{\sg(s)}-1}({\bf x}) =
\sum_{(r_1, \ldots, r_n) \in I(i_1, \ldots, i_n;j_1, \ldots, j_n)} c_{r_1,\ldots, r_n} x_1^{r_1} \cdots x_n^{r_n}$$
for some finite index set $I(i_1, \ldots, i_n;j_1, \ldots, j_n)$ where $c_{r_1,\ldots, r_n}\geq 0$ for all \\
$(r_1, \ldots, r_n) \in I(i_1, \ldots, i_n;j_1, \ldots, j_n)$.
Since $a_n({\bf x}) \geq_{{\bf x}} 0$ and $\mathrm{deg}(a_n({\bf x}))=n$, the degree of any term of the form
$\prod_{s=1}^k a_{i_s-1+j_{\sg(s)}-1}({\bf x})$ is
$$\sum_{s=0}^k i_s-1+j_{\sg(s)}-1 = \sum_{s=0}^k i_s-1+j_{s}-1.$$
Thus the degree of $\mathrm{det}(H[\{i_1,\ldots, i_k\},\{j_1,\ldots, j_k\}])$ less than or equal to  $\sum_{s=0}^k i_s-1+j_{s}-1$.
In particular, $r_1 + \cdots +r_n \leq \sum_{s=0}^k i_s-1+j_{s}-1$ for all
$(r_1,\ldots, r_n) \in I(i_1, \ldots, i_n;j_1, \ldots, j_n)$. But then
$\mathrm{det}({\bf H}[\{i_1,\ldots, i_k\},\{j_1,\ldots, j_k\}])$ equals
\begin{eqnarray*}
&&\sum_{\sg \in S_k} \mathrm{sgn}(\sg) \prod_{s=1}^k x_0^{i_s-1+j_{\sg(s)}-1}a_{i_s-1+j_{\sg(s)}-1}\left(\frac{x_1}{x_0}, \ldots, \frac{x_n}{x_0}\right) \\
&&=x_0^{\sum_{s=0}^k i_s-1+j_{s}-1} \sum_{\sg \in S_k} \mathrm{sgn}(\sg) \prod_{s=1}^k a_{i_s-1+j_{\sg(s)}-1}\left(\frac{x_1}{x_0}, \ldots, \frac{x_n}{x_0}\right) \\
&&=x_0^{\sum_{s=0}^k i_s-1+j_{s}-1} \sum_{(r_1, \ldots, r_n) \in I(i_1, \ldots, i_n;j_1, \ldots, j_n)} c_{r_1,\ldots, r_n} \left(\frac{x_1}{x_0}\right)^{r_1}\cdots \left(\frac{x_n}{x_0}\right)^{r_n}.
\end{eqnarray*}
By our remarks above, $x_0^{\sum_{s=0}^k i_s-1+j_{s}-1} \sum_{(r_1, \ldots, r_n) \in I(i_1, \ldots, i_n;j_1, \ldots, j_n)} c_{r_1,\ldots, r_n} \left(\frac{x_1}{x_0}\right)^{r_1}\cdots \left(\frac{x_n}{x_0}\right)^{r_n}$ is
polynomial in
$x_0,{\bf x}$ with nonnegative coefficients. Thus ${\bf H}(\alpha)$ is $(x_0,{\bf x})$-TP so that $H_{x_0}(\alpha)$ is a Stieltjes moment sequence of polynomials.
\end{proof}
\section{Applications}

In this section, we shall use the results of the previous section
to produce many combinatorially interesting examples
of Stieltjes moment sequences of polynomials.\\
\ \\
{\bf Example 3.1.} Let $\pi =(r_1(q),r_2(q), r_3(q), \ldots ) = (1,1,1,\ldots)$,
$\sigma=(s_0(q),s_1(q), s_2(q), \ldots ) = (1,1+q,1+q,\ldots)$ and
$\tau=(t_1(q),t_2(q),t_3(q), \ldots ) = (q,q,q,\ldots)$.  It is
easy to check that these sequences satisfy the hypothesis
of Lemma \ref{key}. In this case, we are considering
the polynomials defined by
\begin{eqnarray*}
a_{0,0}(q) &=& 1,\\
a_{n+1,0}(q) &=& a_{n,0}(q)+q a_{n,1}(q) \ \mbox{for $n \geq 1$},
\ \mbox{and}\\
a_{n+1,k}(q)&=& a_{n,k-1}(q)+(1+q)a_{n,k}(q)+q a_{n,k+1}(q) \ \mbox{for} \
1 \leq k \leq n.
\end{eqnarray*}
where $a_{n,k}(q)=0$ unless $n\geq k\geq 0$.
In this case, the underlying combinatorial objects are Motzkin paths
where the weights of the up-steps are 1, the weights of the down-steps are
$q$ and the weights of the level-steps are 1 at level 0 and $1+q$ at
levels $k > 0$. Thus we can interpret $a_{n,k}(q)$ as  the sum of the weights
of Motzkin paths that start at $(0,0)$ and end at $(n,k)$.
For example, if $A(q) = [a_{n,k}(q)]$, then
\begin{equation*}\
  A(q)=\left[
      \begin{array}{cccccc}
        1 &  &  &  &  &  \\
        1 & 1 &   &   &   &   \\
        1+q & 2+q & 1 &   &   &   \\
        1+3q+q^2 & 3+5q+q^2 & 3+2q & 1 &   &   \\
        \vdots & \vdots & \vdots& \vdots &  & \ddots \\
      \end{array}
    \right].
\end{equation*}

\begin{proposition}
The sequence $(a_{n,0}(q))_{n\ge 0}$ is a
Stieltjes moment sequence of polynomials.
\end{proposition}

A {\it Riordan array}, denoted by $(d(x),h(x))$, is an infinite lower triangular matrix
whose generating function of the $k$th column is $x^kh^k(x)d(x)$ for $k=0,1,2,\ldots$,
where $d(0)=1$ and $h(0)\neq 0$ \cite{SGWW91}.
A Riordan array $R=[r_{n,k}]_{n,k\ge 0}$
can be characterized by two sequences
$(a_n)_{n\ge 0}$ and $(z_n)_{n\ge 0}$ such that
\begin{equation}\label{rrr-c}
r_{0,0}=1,\quad r_{n+1,0}=\sum_{j\ge 0}z_jr_{n,j},\quad r_{n+1,k+1}=\sum_{j\ge 0}a_jr_{n,k+j}
\end{equation}
for $n,k\ge 0$ (see \cite{HS09} for instance).
Call $(a_n)_{n\ge 0}$ and $(z_n)_{n\ge 0}$
the $A$- and $Z$-sequences of $R$ respectively.
Let $Z(x)=\sum_{n\ge 0}z_nx^n$ and $A(x)=\sum_{n\ge 0}a_nx^n$
be the generating functions of $(z_n)_{n\ge 0}$ and $(a_n)_{n\ge 0}$ respectively.
Then it follows from \eqref{rrr-c} that
\begin{equation}\label{fg}
  d(x)=\frac{1}{1-xZ(xh(x))} \ \mbox{and} \ h(x)=A(xh(x)).
\end{equation}

The matrix $R(a,b;c,e)=[r_{n,k}]_{n,k\ge 0}$, where
\begin{equation}\label{rr}
\left\{
  \begin{array}{ll}
    r_{0,0}=1,\quad r_{n+1,0}=ar_{n,0}+br_{n,1},\\
    r_{n+1,k+1}=r_{n,k}+cr_{n,k+1}+er_{n,k+2},
  \end{array}
\right.
\end{equation}
is called the recursive matrix.
The coefficient matrix of the recursive matrix \eqref{rr} is defined to
be
\begin{equation}\label{J-pqst}
J(p,q;s,t)=\left[
\begin{array}{ccccc}
a & 1 &  &  &\\
b & c & 1 &\\
 & e & c & 1 &\\
 & & e & c & \ddots\\
& & &\ddots & \ddots \\
\end{array}\right].
\end{equation}

Now $R(a,b;c,e)$ is a Riordan array with $Z(x)=a+bx$ and $A(x)=1+cx+ex^2$.
Let $R(a,b;c,e)=(d(x),h(x))$.
Then by \eqref{fg}, we have
$$d(x)=\frac{1}{1-x(a+bxh(x))} \ \mbox{and} \ h(x)=1+cxh(x)+ex^2h^2(x).$$
It follows that
$$h(x)=\frac{1-cx-\sqrt{1-2cx+(c^2-4e)x^2}}{2ex^2}$$
and
$$d(x)=\frac{2e}{2e-b+(bc-2ae)x+b\sqrt{1-2cx+(c^2-4e)x^2}}$$
(see \cite{WZ-LAA} for details).
From this formula we can now derive a number of interesting examples.
For example, taking $a=1$, $b=q$, $c=1+q$ and $e=q$ in \eqref{J-pqst},
we obtain the generating function of the $(a_{n,0}(q))$ is
$$d_A(x,q)= \sum_{n \geq 0} a_{n,0}(q) x^n =
\frac{2}{1+(q-1)x+\sqrt{1-2(1+q)x+(1-q)^2x^2}}.$$

\begin{rem}
\begin{enumerate}
\item When we set $q =1$ in $A(q)$, we obtain the
Catalan triangle of Aigner~\cite{Aig99}.
See also sequence \cite[A039599]{OEIS} in the On-line Encyclopedia
of Integer Sequences.
It follows that
$a_{n,0}(1) =C_n$ where $C_n = \frac{1}{n+1}\binom{2n}{n}$ is the
$n$-th Catalan number \cite[A000108]{OEIS}. Hence
$a_{n,0}(q)$ is a $q$-analogue of the Catalan number $C_n$.

\item When we set $q =2$ in $A(q)$, we obtain the triangle
\cite[A172094]{OEIS} and  $a_{n,0}(2)$ are the little Schr\"oder numbers
$S_n$ \cite[A001003]{OEIS}.
It follows that
$a_{n,0}(2q)$ is a $q$-analogue of $n$-th little Schr\"oder number $S_n$.

\item When we set $q=3$, the sequence $(a_{n,0}(3))_{n \geq 0}$ is sequence
 \cite[A007564]{OEIS}. It follows that
$a_{n,0}(3q)$ is a $q$-analogue of the sequence \cite[A007564]{OEIS}.

\item When we set $q=4$, the sequence $(a_{n,0}(4))_{n \geq 0}$ is sequence
 \cite[A059231]{OEIS}. It follows that
$a_{n,0}(4q)$ is a $q$-analogue of the sequence \cite[A059231]{OEIS}.

\end{enumerate}
\end{rem}

We know that for any $m \geq 1$, the sequence $(a_{n,0}(m))_{n \geq 0}$ is
a Stieltjes moment sequence. In particular, the
Catalan numbers $C_n$ and the little Schr\"oder numbers $S_n$ are a Stieltjes moment sequences. \\
\ \\
{\bf Example 3.2.} Let
\begin{eqnarray*}
\pi &=&(r_1(q),r_2(q),r_3(q), \ldots ) = (1,1,1,\ldots),\\
\sigma&=& (s_0(q),s_1(q),s_2(q), \ldots ) = (1+q+q^2,1+q+q^2,1+q+q^2,\ldots),
\ \mbox{and} \\
\tau&=&(t_1(q),t_2(q),t_3(q), \ldots ) = (q,q,q,\ldots).
\end{eqnarray*}
  It is
easy to check that these sequences satisfy the hypothesis
of Lemma \ref{key}. In this case, we are considering
the polynomials defined by
\begin{eqnarray*}
b_{0,0}(q) &=& 1, \\
b_{n+1,0}(q) &=& (1+q+q^2) b_{n,0}(q)+q b_{n,1}(q) \ \mbox{for $n \geq 1$}, \mbox{and} \\
b_{n+1,k}(q)&=&b_{n,k-1}(q)+(1+q+q^2)b_{n,k}(q)+q b_{n,k+1}(q) \ \mbox{for
$1 \leq k \leq n$}
\end{eqnarray*}
where $b_{n,k}(q)=0$ unless $n\ge k\ge 0$.
In this case, the underlying combinatorial objects are Motzkin paths
where the weights of the up-steps are 1, the weights of the down-steps are
$q$, and the weights of the level-steps $1+q+q^2$. Thus we can interpret $b_{n,k}(q)$ as  the sum of the weights
of Motzkin paths that start at $(0,0)$ and end at $(n,k)$.
For example, if $B(q) = [b_{n,k}(q)]$, then
\begin{equation*}\
  B(q)=\left[
      \begin{array}{cccccc}
        1 &  &  &  &  &  \\
        1 +q +q^2 & 1 &   &   &   &   \\
        1+3q+3q^2 + 2q^3 +q^4& 2+2q +2q^2 & 1 &   &   &   \\
        \left( 1+6q+9q^2+10q^3+\right.& 3+8q+9q^2 +6q^3+3q^4
& 3+3q +3q^2& 1 &   &   \\
\left. 6q^4+3q^5+q^6 \right) & & & & & \\
\vdots & \vdots & \vdots& \vdots &  & \ddots \\
      \end{array}
    \right].
\end{equation*}

\begin{proposition}
The sequence $(b_{n,0}(q))_{n\ge 0}$ is a
Stieltjes moment sequence of polynomials.
\end{proposition}

Taking $a=1+q+q^2$, $b=q$, $c=1+q+q^2$ and $e=q$ in \eqref{J-pqst},
we obtain the generating function of the $(b_{n,0}(q))$ is
$$d_B(x,q)= \sum_{n \geq 0} b_{n,0}(q)x^n =
\frac{2}{1-(1+q+q^2)x+\sqrt{1-2(1+q+q^2)x+((1+q+q^2)^2-4q)x^2}}.$$

\begin{rem}
In this case, the triangle $B(1)$ is \cite[A091965]{OEIS} and the
first column $(b_{n,0}(1))_{n \geq 0}$ is sequence \cite[A002212]{OEIS}.
Clearly $b_{n,0}(1)$ counts the number of 3-colored Motzkin paths of
length $n$ and the number of restricted hexagonal polyominoes with
$n$ cells.
\end{rem}

It follows that for any $m \geq 1$, the sequence $(d_{n,0}(m))_{n \geq 0}$ is
a Stieltjes moment sequence. In particular, the sequence which
counts the number of restricted hexagaonal polynominoes
is a Stieltjes momoment sequence. \\
\ \\
{\bf Example 3.3.} Let
\begin{eqnarray*}
\pi &=&(r_1(p,q),r_2(p,q),r_3(p,q), \ldots ) = (1,1,1,\ldots), \\
\sigma&=& (s_0(p,q),s_1(p,q),s_2(p,q), \ldots ) = (1+p+q,1+p+q,1+p+q,\ldots),
\ \mbox{and} \\
\tau &=&(t_1(p,q),t_2(p,q),t_3(p,q), \ldots ) = (q,q,q,\ldots).
\end{eqnarray*}  It is
easy to check that these sequences satisfy the hypothesis
of Lemma \ref{key}. In this case, we are considering
the polynomials defined by
\begin{eqnarray*}
c_{0,0}(p,q) &=& 1,\\
c_{n+1,0}(p,q) &=& (1+p+q) c_{n,0}(p,q)+q c_{n,1}(p,q) \ \mbox{for $n \geq 1$}, \
\mbox{and} \\
c_{n+1,k}(p,q)&=&c_{n,k-1}(p,q)+(1+p+q)c_{n,k}(p,q)+q c_{n,k+1}(p,q) \ \mbox{for
$1 \leq k \leq n$}
\end{eqnarray*}
where $c_{n,k}(p,q)=0$ unless $n\ge k\ge 0$.
In this case, the underlying combinatorial objects are Motzkin paths
where the weights of the up-steps are 1, the weights of the down-steps are $q$, and the weights of the level-steps $1+p+q$. Thus we can interpret $c_{n,k}(p,q)$ as  the sum of the weights
of Motzkin paths that start at $(0,0)$ and end at $(n,k)$. In particular,
we can interpret $c_{n,}(p,q)$ as weighted sum over three
colored Motzkin paths.
That is, the levels of the Motzkin path can be colored with one of
three colors, namely, color 0 which has weight 1, color 1 which has weight $q$,
and color 2 which has weight $p$, and the down-steps have weight $q$.
For example, if $C(p,q) = [c_{n,k}(p,q)]$, then
\begin{equation*}\
  C(p,q)=\left[
      \begin{array}{cccccc}
        1 &  &  &  &  &  \\
        1 +p+q & 1 &   &   &   &   \\
        1+3p+2q +2pq+p^2+q^2 & 2+2p+2q & 1 &   &   &   \\
        \left(1+6p+3q+6p^2+9pq+3q^2+\right. & \left( 3+8p+6q+\right. & 3+3p+3q & 1 &   &   \\
         \left. p^3+3p^2q+3pq^2+q^3\right)   &\left. 3p^2+6pq+3q^2 \right)                    &         &   &   & \\
\vdots & \vdots & \vdots&\vdots  &  & \ddots \\
      \end{array}
    \right].
\end{equation*}

\begin{proposition}
The sequence
$(c_{n,0}(p,q))_{n \geq 0}$ is a
Stieltjes moment sequence of polynomials.
\end{proposition}
Taking $a=1+p+q$, $b=q$, $c=1+p+q$ and $e=q$ in \eqref{J-pqst},
we obtain the generating function of the $c_{n,0}(p,q)$s is
$$d_C(x,p,q)=\sum_{n \geq 0} c_{n,0}(p,q)x^n = \frac{2}{1-(1+p+q)x+\sqrt{1-2(1+p+q)x+((1+p+q)^2-4q)x^2}}.$$

\begin{rem}
\begin{enumerate}
\item When we set $p=q =1$ in $(c_{n,0}(1,1))_{n \geq 0}$, we obtain the
$1,3,10,36,137, \ldots $  which is sequence
\cite[A002212]{OEIS}. Besides counting 3-colored Motzkin path, it also
the number of restricted hexagonal polyominoes with $n$-cells.

\item When we set $p=1$ and $q =2$ in $(c_{n,0}(p,q))_{n \geq 0}$, we obtain the
$1,4,18,88,456,2464, \ldots $ which is sequence \cite[A024175]{OEIS}.

\item When we set $p=2$ and $q =2$ in $(c_{n,0}(p,q))_{n \geq 0}$, we obtain the
$1,4,20,112,672,4224,  \ldots $ which is sequence \cite[A003645]{OEIS}
whose $n$-th term is $2^nC_{n+1}$.
\end{enumerate}
\end{rem}

There are many variations of the $t_k(p,q)$-sequence that will also
produce Steiltjes moment sequence of polynomials.
For example, suppose we define $t^{(s)}_k(p,q)$ to be $q$ if $k \leq s$
and $p$ if $k > s$ and let $\tau^{(s)} = (t^{(s)}_1(p,q),t^{(s)}_2(p,q),t^{(s)}_3(p,q), \ldots)$.
It is easy to see the sequences $\pi =(r_1(p,q),r_2(p,q),r_3(p,q), \ldots ) = (1,1,1,\ldots)$,
$\sigma=(s_0(p,q),s_1(p,q), s_2(p,q),\ldots ) = (1+p+q,1+p+q,1+p+q,\ldots)$ and
$\tau^{(s)}$ satisfy the hypothesis
of Lemma \ref{key} for all $s$. Then we can define the polynomials
$c^{(s)}_{n,k}(p,q)$ by
\begin{eqnarray*}
c^{(s)}_{0,0}(p,q) &=& 1,\\
c^{(s)}_{n+1,0}(p,q) &=& (1+p+q) c^{(s)}_{n,0}(p,q)+t^{(s)}_1(p,q) c^{(s)}_{n,1}(p,q)
\ \mbox{for $n \geq 1$}, \ \mbox{and} \\
c^{(s)}_{n+1,k}(p,q)&=&c^{(s)}_{n,k-1}(p,q)+(1+p+q)c^{(s)}_{n,k}(p,q)+t^{(s)}_{k+1}(p,q) c^{(s)}_{n,k+1}(p,q) \ \mbox{for
$1 \leq k \leq n$}
\end{eqnarray*}
where $c^{(s)}_{n,k}(p,q)=0$ unless $n\ge k\ge 0$.

\begin{proposition} For all $s \geq 0$,
$(c^{(s)}_{n,0}(p,q))_{n \geq 0}$ is a
Stieltjes moment sequence of polynomials.
\end{proposition}

One of the advantages of this set up  is that we can set $p=0$ in
such sequences.  In particular,
$(c^{(s)}_{n,0}(0,q))_{n \geq 0}$ is a
Stieltjes moment sequence of polynomials. In such a situation,
$c^{(s)}_{n,0}(0,q)$ is the sum over the weights of 2-colored Motzkin paths
of height $\leq s$. That is, the level steps can be colored with
color 0 which has weight 1 or colored with color 1 which has weight $q$.
The down-steps all have weight $q$ and the up-steps all have weight 1.

We can also generalize this example by adding more variables. That is, let
${\bf x} = ({\bf x})$ where $n \geq 3$ and let
$1 \leq s_1 < \cdots < s_{n-1}$.  Then let
$r_i({\bf x}) =1$ for all $i \geq 1$,
$s_i({\bf x}) =1+x_1+ \cdots + x_n$ for all $i \geq 1$, and
$t^{(s_1, \ldots, s_{n-1})}_i({\bf x})$ equal $x_1$ if $i \leq s_1$, $x_j$ if
$s_{j-1} < i \leq s_j$, and $x_n$ if $i > s_{n-1}$. Then
let $\pi =(r_1({\bf x}),r_2({\bf x}),r_3({\bf x}), \ldots ) = (1,1,1,\ldots)$,
$\sigma=(s_0({\bf x}),s_1({\bf x}),s_2({\bf x}), \ldots )$ and
$\tau^{(s_1, \ldots, s_{n-1})}=(t^{(s_1, \ldots, s_{n-1})}_1({\bf x}),t^{(s_1, \ldots, s_{n-1})}_2({\bf x}), t^{(s_1, \ldots, s_{n-1})}_3({\bf x}),\ldots )$.  It is
easy to check that for any $1 \leq s_1 < \cdots < s_{n-1}$, these
sequences satisfy the hypothesis
of Lemma \ref{key}. In this case, we are considering
the polynomials defined by \\
\ \\
$c^{(s_1, \ldots, s_{n-1})}_{0,0}({\bf x}) =1$, \\
$c^{(s_1, \ldots, s_{n-1})}_{n+1,0}({\bf x}) = (1+x_1+ \cdots +x_n) c^{(s_1, \ldots, s_{n-1})}_{n,0}({\bf x})+t^{(s_1, \ldots, s_{n-1})}_1({\bf x}) c^{(s_1, \ldots, s_{n-1})}_{n,1}({\bf x})$ for $n \geq 1$, and \\
$c^{(s_1, \ldots, s_{n-1})}_{n+1,k}({\bf x})=
c^{(s_1, \ldots, s_{n-1})}_{n,k-1}({\bf x})+(1+x_1 + \cdots
+x_n)c^{(s_1, \ldots, s_{n-1})}_{n,k}({\bf x})+
t^{(s_1, \ldots, s_{n-1})}_{k+1} c^{(s_1, \ldots, s_{n-1})}_{n,k+1}({\bf x})$\\ for
$1 \leq k \leq n$, \\
\ \\
where $c^{(s_1, \ldots, s_{n-1})}_{n,k}({\bf x})=0$ unless $n\ge k\ge 0$.
In this case, the underlying combinatorial objects are Motzkin paths
where the weights of up-steps are 1, the weights of down-steps
ending at level $k$ are $t^{(s_1, \ldots, s_{n-1})}_{k+1}({\bf x})$,
and the weights of level-steps
$1+x_1 + \cdots + x_n$. Thus we can interpret $c_{n,k}({\bf x})$ as  the sum of the weights
of Motzkin paths that start at $(0,0)$ and end at $(n,k)$. In particular,
we can interpret $c_{n,0}({\bf x})$ as weighted sum over $(n+1)$-colored Motzkin paths.
That is, the levels of the Motzkin path can be colored with one of
$(n+1)$-colors, namely, color 0 which has weight 1, color i which has weight $x_i$ for
$i =1, \ldots, n$,  and the down-steps that end at level
$k$ have weight $t^{(s_1, \ldots, s_{n-1})}_{k+1}({\bf x})$.

\begin{proposition} For all $1 \leq s_1 < \cdots < s_{n-1}$,
$(c^{(s_1, \ldots, s_{n-1})}_{n,0}({\bf x}))_{n \geq 0}$ is a
Stieltjes moment sequence of polynomials.
\end{proposition}

Once again for any $2 \leq j \leq x_n$, the polynomials
$(c^{(s_1, \ldots, s_{n-1})}_{n,0}(x_1, \ldots,x_{j-1},0, \ldots, 0))_{n \geq 0}$
Stieltjes moment sequence of polynomials and
$c^{(s_1, \ldots, s_{n-1})}_{n,0}(x_1, \ldots,x_{j-1},0, \ldots, 0)$ equals
to the sum of $j$-colored Motzkin paths of length $n$ and  height less that $s_j$.\\
\ \\
{\bf Example 3.4.}
Let $[n]_{p,q} = \frac{p^n-q^n}{p-q}= p^{n-1}+qp^{n-2} + \cdots + q^{n-2}p +q^{n-1}$. Let
\begin{eqnarray*}
\pi^{(u)} &=&(r_1(p,q),r_2(p,q),r_3(p,q),\ldots)=(1,1,1,\ldots), \\
\sigma^{(u)}&=&(s_0(p,q),s_1(p,q),s_2(p,q), \ldots ) = ([u]_{p,q},p+q,p+q,\ldots), \ \mbox{and} \\
\tau^{(u)}&=&(t_1(p,q),t_2(p,q), t_3(p,q), \ldots ) = (pq[u-1]_{p,q},pq,pq,\ldots),
\end{eqnarray*}
where $u \geq 3$.
In this case, we are considering the polynomials defined by
\begin{eqnarray*}
d^{(u)}_{0,0}(p,q) &=& 1, \\
d^{(u)}_{n+1,0}(p,q) &=& [u]_{p,q}d^{(u)}_{n,0}(p,q)+pq [u-1]_{p,q}d^{(u)}_{n,1}(p,q) \ \mbox{for $n \geq 1$}, \ \mbox{and} \\
d^{(u)}_{n+1,k}(p,q)&=&d^{(u)}_{n,k-1}(p,q)+(p+q)d^{(u)}_{n,k}(p,q)+pq d^{(u)}_{n,k+1}(p,q) \ \mbox{for
$1 \leq k \leq n$}
\end{eqnarray*}
where $d^{(u)}_{n,k}(p,q)=0$ unless $n\ge k\ge 0$.
In this case, the underlying combinatorial objects are Motzkin paths
where the weights of the up-steps are 1, the weights of the down-steps
that ends
a level 0 are $pq[u-1]_{p,q}$ and the weights of the down-steps that end
at level $k > 0$ are $pq$, and the weights of the
level-steps is $[u]_{p,q}$ if the step is at
level 0 and $p+q$ if the step is at level $k \geq 1$. Thus we can interpret $d^{(u)}_{n,k}(p,q)$ as  the sum of the weights
of Motzkin paths that start at $(0,0)$ and end at $(n,k)$.
For example, if $D^{(3)}(p,q) = [e_{n,k}(p,q)]$, then
$$
  D^{(3)}(p,q)=\left[
      \begin{array}{cccccc}
        1 &  &  &  &  &  \\
        p^2 +pq +q^2 & 1 &   &   &   &   \\
        (p^2 +pq +q^2)^2+pq(p+q) & p^2 +pq +q^2+p+q & 1  &   &   \\
        \vdots & \vdots &\vdots &  &  & \ddots \\
      \end{array}
    \right].
$$

In this case,
the sequences do not satisfy the hypothesis of both Lemma \ref{key}
and Lemma \ref{product}.
Nevertheless,
we can prove directly that the tridiagonal matrix
$J^{(u)}:=J^{(\pi,\sg,\tau)}(p,q)$ where
$$J^{(u)}=\begin{bmatrix} [u]_{p,q} & 1 &  &  &   \\
pq[u-1]_{p,q} & p+q & 1 & \\
&  pq & p+q & 1 &  \\
 &   & pq & p+q & 1 &  \\
 &  &&\ddots&\ddots&\ddots
\end{bmatrix}$$
is $p,q$-TP. By Lemma \ref{basic}, we
need only show that principal minors of the form \\
$\mathrm{det}J^{(u)}[\{n,n+1, \ldots, n+k-1\},\{n,n+1, \ldots, n+k-1\}]$ are polynomials
with nonnegative coefficients.

For minors of the form $\mathrm{det}J^{(u)}[\{n,n+1, \ldots, n+k-1\},\{n,n+1, \ldots, n+k-1\}]$
where $n > 1$, we are dealing with the tridiagonal matrix
$$L=\begin{bmatrix} p+q & 1 &  &  &   \\
pq & p+q & 1 & \\
&  pq & p+q & 1 &  \\
 &   & pq & p+q & 1 &  \\
 &  &&\ddots&\ddots&\ddots
\end{bmatrix}$$
which does satisfy the hypothesis of  Lemma \ref{product}
where $b_i(p,q) = p$ and $c_i(p,q) =q$ for all $i \geq 1$.
Thus we need only
consider minors of the form
$N^{(u)}_k(p,q) = \mathrm{det}J^{(u)}[\{1,2, \ldots, k\},\{1,2, \ldots, k\}]$.
We shall
prove by induction that $N^{(u)}_k(p,q) =[k+u-1]_{p,q}$.  Note that $p[n]_{p,q} = p^n+pq[n-1]_{p,q}$ and $q[n]_{p,q} = pq[n-1]_{p,q}+q^n$.
Hence
$N^{(u)}_1 =[u]_{p,q}$ and
$$
N^{(u)}_2 =(p+q)[u]_{p,q} -pq[u-1]_{p,q} = p^u+pq[u-1]_{p,q}+pq[u-1]_{p,q}+q^u -pq[u-1]_{p,q} = [u+1]_{p,q}.
$$

Now assume that $k \geq 3$.
Then by expanding the determinant about the last row of $J^{(u)}[\{1,2, \ldots, k\},\{1,2, \ldots, k\}]$, we see that
\begin{eqnarray*}
N^{(u)}_k &=&(p+q)N_{k-1} -pqN^{(u)}_{k-2} = (p+q)[k-1+u-1]_{p,q} -pq [k-2+u-1]_{p,q} \\
&=& p^{k+u-2} + pq[k+u-3]_{p,q} +pq[k+u-3]_{p,q}+q^{k+u-2} -pq[k+u-3]_{p,q} = [k+u-1]_{p,q}.
\end{eqnarray*}
Thus $J^{(u)}$ is $q$-TP for all $u \geq 3$.

Thus we can apply Theorem \ref{HMJ} to obtain the following result.

\begin{proposition} For all $ u\geq 3$, $(d^{(u)}_{n,0}(p,q))_{n \geq 0}$
is a Stieltjes moment sequence of polynomials.
\end{proposition}
Taking $a=[u]_{p,q}$, $b=pq[u-1]_{p,q}$, $c=p+q$ and $e=pq$ in \eqref{J-pqst},
we obtain the generating function of the $(d^{(u)}_{n,0})$ is
\begin{multline*}
d_{D,u}(x,p,q)= \sum_{n \geq 0} d^{(u)}_{n,0} x^n = \\
\frac{2}{2-[u-1]_{p,q}+([u-1]_{p,q}(p+q)-2[u]_{p,q})x+[u-1]_{p,q}\sqrt{1-2(p+q)x+(p-q)^2x^2}}.
\end{multline*}

\begin{rem}
Setting $p=q=1$ in $(d^{(3)}_{n,0}(p,q))_{n \geq 0}$ gives the sequence
$1,3,11,43,173,707,2917,\ldots $ which is sequence
\cite[A026671]{OEIS}.  The combinatorial interpretation
for the sequence is the number of lattice paths from $(0,0)$ to $(n,n)$
using steps $(1,0)$, $(0,1)$, and $(1,1)$ when the step is on the diagonal.
\end{rem}\ \\
{\bf Example 3.5.} Let
\begin{eqnarray*}
\pi &=&(r_1(p,q,r),r_2(p,q,r),r_3(p,q,r), \ldots ) = (1,1,1,\ldots), \\
\sigma &=&(s_0(p,q,r),s_1(p,q,r), s_2(p,q,r),\ldots )=
(q+r,p+q+r,p+q+r,\ldots ), \ \mbox{and} \\
\tau&=&(t_1(p,q,r),t_2(p,q,r),t_3(p,q,r), \ldots ) = (q(p+r),q(p+r),q(p+r),\ldots).
\end{eqnarray*}
In this case, we are considering
the polynomials defined by
\begin{eqnarray*}
i_{0,0}(p,q,r) &=& 1,\\
i_{n+1,0}(p,q,r) &=& (q+r)i_{n,0}(p,q,r)+q(p+r) i_{n,1}(p,q,r) \ \mbox{for $n \geq 1$}, \mbox{and} \\
i_{n+1,k}(p,q,r)&=&i_{n,k-1}(p,q,r)+(p+q+r)i_{n,k}(p,q,r)+q(p+r) i_{n,k+1}(p,q,r) \ \mbox{for
$1 \leq k \leq n$}
\end{eqnarray*}
where $i_{n,k}(p,q,r)=0$ unless $n\ge k\ge 0$.
In this case, the underlying combinatorial objects are Motzkin paths
where the weights of the up-steps are 1, the weights of the down-steps are
$q(p+r)$, and the weights of the
level steps at level $0$ is $q+r$ and the weights of
the level steps at level $k \geq 1$ are $p+q+r$. Thus we can interpret $i_{n,k}(p,q,r)$ as  the sum of the weights
of Motzkin paths that start at $(0,0)$ and end at $(n,k)$.
For example, if $I(p,q,r) = [i_{n,k}(q)]$, then
\begin{multline*}\
  I(p,q,r)=\\
\left[
      \begin{array}{cccccc}
        1 &  &  &  &  &  \\
        q+r  & 1 &   &   &   &   \\
        (pq+q^2)+3qr+r^2& (p+2q)+2r & 1 &   &   &   \\
        \left(p^2q+3pq^2+q^3)+\right. & (p^2+5pq+3q^2)+(3p+8q)r+3r^2& (2p+3q) +3r & 1 &   &  \\
        \left.(4pq+6q^2)r +6qr^2+r^3 \right. & & & & & \\
        \vdots & \vdots &\vdots &\vdots  &  & \ddots \\
        \end{array}
    \right].
\end{multline*}

In this case,
the sequences do not satisfy the hypothesis of both Lemma \ref{key} and Lemma \ref{product}.
However,
we can prove directly that the tridiagonal matrix
$J:=J^{(\pi,\sg,\tau)}$ where
$$J=\begin{bmatrix}  q+r & 1 &  \\
 q(p+r) & p+q+r & 1 &  \\
  &  q(p+r) & p+q+r& 1 &  \\
  &  &  q(p+r) & p+q+r & 1 &  \\
  &&&\ddots&\ddots&\ddots
\end{bmatrix}$$
is $(p,q,r)$-TP. By Lemma \ref{basic}, we
need only show that principal minors of the form
$\mathrm{det}J[\{n,n+1, \ldots, n+k-1\},\{n,n+1, \ldots, n+k-1\}]$ are polynomials
with nonnegative coefficients.

For consecutive minors of the form $\mathrm{det}J[\{n,n+1, \ldots, n+k-1\},\{n,n+1, \ldots, n+k-1\}]$  where $n \geq 2$, are considering the matrix
$$\overline{J}=\begin{bmatrix}  p+r+q & 1 &  \\
 q(p+r) & p+r+q & 1 &  \\
  &  q(p+r) & p+r+q& 1 &  \\
  &  &  q(p+r) & p+r+q & 1 &  \\
  &&&\ddots&\ddots&\ddots
\end{bmatrix}.$$
This matrix satisfies the hypothesis of Lemma \ref{product} where
$(b_1(p,q,r),b_2(p,q,r),b_3(p,q,r), \ldots ) = (p+r,p+r,p+r, \ldots)$ and
$(c_1(p,q,r),c_2(p,q,r),c_3(p,q,r),\ldots ) = (q,q,q \ldots)$.

Thus we need only consider minors of the form
$P_k = \mathrm{det}J[\{1,2, \ldots, k\},\{1,2, \ldots, k\}]$. We shall
prove by induction that $P_k = q^n + \sum_{k=1}^n r^k \left( \sum_{j=0}^k \binom{n-k-j+k-1}{k-1} q^j p^{n-k-j}\right).$
Clearly
$P_1 =q+r$ and $P_2 =(q+r)(p+q+r) -q(p+r) = q^2 +(p+q)r+r^2$. Now assume that $k \geq 3$.
Then by expanding the determinant about the last row of $J[\{1,2, \ldots, k\},\{1,2, \ldots, k\}]$, we see that
$$P_k =(p+q+r)P_{k-1} -q(p+r)P_{k-2}.$$

In particular,
$$P_k|_{r^0} =(p+q)P_{k-1}|_{r^0} -qpP_{k-2}|_{r^0} =(p+q)q^{k-1}-qpq^{k-2} = q^k.$$
Similarly,
$$P_k|_{r^n} =r*P_{k-1}|_{r^{n}} -q(p+r)P_{k-2}|_{r^n} =r*r^{n-1} -0  = r^k.$$

Now suppose that $1 \leq r \leq n-1$.
Then we must show that
\begin{equation}\label{ind}
P_k|_{r^kq^j}  =(p+q)P_{k-1}|_{r^kq^j} + P_{k-1}|_{r^{k-1}q^j}-pqP_{k-2}|_{r^kq^j} - qP_{k-2}|_{r^{k-1}q^j}.
\end{equation}
One can show that this yields a identity among binomial coefficients which can be directly checked with Mathematica.
Thus $J$ is $p,q,r$-TP.

Thus we can apply Theorem \ref{HMJ} to obtain the following result.

\begin{proposition}
The sequence $(i_{n,0}(p,q,r))_{n \geq 0}$
is a
Stieltjes moment sequence of polynomials.
\end{proposition}

On the other hand, taking $a=q+r$, $b=q(p+r)$, $c=p+q+r$ and $e=q(p+r)$ in \eqref{J-pqst},
we obtain the generating function of the first column $I(p,q,r)$ is
$$d(x,p,q,r)=\frac{2}{1+(p-q-r)x+\sqrt{1-2(p+q+r)x+((p+q+r)^2-4q(p+r))x^2}}.$$

Here is a list of the first few values of the sequence $(i_{n,0}(p,q,r))_{n \geq 0}$.

$1$\\
$q+r$\\
$(pq+q^2)+3 q r+r^2$\\
$(p^2 q+3 p q^2+q^3)+\left(4 p q+6 q^2\right) r+6 q r^2+r^3$\\
$(p^3 q+6 p^2 q^2+6 p q^3+q^4)+\left(5 p^2 q+20 pq^2+10 q^3\right) r+\left(10 p q+20 q^2\right) r^2+10 q r^3+r^4$

\begin{rem}
\begin{enumerate}
\item It follows from earlier results that $i_{n,0}(p,q,0) = \sum_{k=1}^{n} \frac{1}{k} \binom{n-1}{k-1} \binom{n}{k-1} q^kp^{n-k}$ from which it follows
that $i_{n,0}(1,1,0) = C_n$ where $C_n = \frac{1}{n+1}\binom{2n}{n}$ is the
$n$-th Catalan number.

\item We can show that $i_{n,0}(p,q,1) = \sum_{k=1}^{n} \frac{1}{k+1} \binom{n+k}{k} \binom{n}{k} q^kp^{n-k}$ from which it follows that
$(i_{n,0}(1,1,1))_{n \geq 0}$ is sequence \cite[A006318]{OEIS} which is the sequence of large Sch\"{o}oder numbers.

\item We can show that  $(i_{n,0}(1,1,r))_{n \geq 0}$ is the triangle  \cite[A060693]{OEIS}.

\item The sequence $(i_{n,0}(1,1,2))_{n \geq 0}$ starts out $1,3,12,57,300,1686,9912, \ldots$. This is sequence \cite[A047891]{OEIS}.

\item The sequence $(i_{n,0}(1,2,1))_{n \geq 0}$ starts out $1,3,13,67,381,2307,14598, \ldots$. This is sequence \cite[A064062]{OEIS} which is
the sequence of generalized Catalan numbers.
\end{enumerate}

\end{rem}

Next, in Examples 3.6-3.9, we will consider several
sequences of polynomials in $q$
that were studied by Zhu in \cite{Zhu13}.  In each case, Zhu
showed that the sequence of polynomials is strongly $q$-log convex sequence.
Chen et al. \cite{Chen10} proved that Narayana polynomials form
a strongly $q$-log convex sequence.
Chen et al. \cite{Chen11} also proved that Bell polynomials form
a strongly $q$-log convex sequence.
Wang and Zhu \cite{WZ16} further studied the polynomials of Zhu and
showed that each sequence has the stronger property of being
Stieltjes moment sequences of polynomials. Our results will
show that there are natural
$(p,q)$-analogues of these polynomials which
are Stieltjes moment sequences of polynomials. \\
\ \\
{\bf Example 3.6.}
Wang and Zhu \cite{WZ16} proved that the sequence of Narayana polynomials
$(W_n(q))_{n \geq 0}$ of type $B$ is a Stieltjes moment sequence of
 where $W_n(q) = \sum_{k=0}^n \binom{n}{k}^2q^k$.
It follows from Theorem \ref{H} that the sequence $(W_n(p,q))_{n \geq 0}$
where $W_n(p,q) = \sum_{k=0}^n \binom{n}{k}^2q^kp^{n-k}$
is a Stieltjes moment sequence of polynomials.

\begin{proposition}
The sequence $\left(\sum_{k=0}^{n} \binom{n}{k}^2 p^kq^{n-k}\right)_{n \geq 0}$
is a Stieltjes moment sequence of polynomials.
\end{proposition}

Taking $a=p+q$, $b=2pq$, $c=p+q$ and $e=pq$ in \eqref{J-pqst},
we obtain the generating function of the $e_{n,0}(p,q)$s is
$$d(x,p,q)=
\frac{1}{\sqrt{1-2(p+q)x+(p-q)^2x^2}}.$$

\begin{rem}
\begin{enumerate}
\item When we set $p=q =1$ in $(e_{n,0}(p,q))_{n \geq 0}$, we obtain the
$1,2,6,20,70,252, \ldots $ which is sequence  \cite[A001850]{OEIS}
which are central binomial coefficients $\binom{2n}{n}$.  Thus
we can view $(e_{n,0}(p,q))_{n \geq 0}$ as a $(p,q)$-analogue of
the central binomial coefficients.

\item When we set $p=2$ and $q =1$ in $(e_{n,0}(p,q))_{n \geq 0}$, we obtain the
$1,3,13,63,321,1683, \ldots $ which is sequence  \cite[A000984]{OEIS}
which are central Delannoy numbers.  Thus
we can view $(e_{n,0}(2p,q))_{n \geq 0}$ as a $(p,q)$-analogue of
the central Delannoy numbers.

\item When we set $p=2$ and $q =2$ in $(e_{n,0}(p,q))_{n \geq 0}$, we obtain the
$1,4,24,160,1120,8064, \ldots $ which is sequence \cite[A059304]{OEIS}.
This $n$-th term of this sequence also counts the number
of paths from (0,0) to $(n,n)$ using steps (0,1)
and two kinds of steps (1,0).

\item When we set $p=2$ and $q =4$ in $(e_{n,0}(p,q))_{n \geq 0}$, we obtain the
$1,6,52,504,5136,\ldots $ which is sequence  \cite[A084773]{OEIS}.
This sequence
has a interpretation in terms of weighted Motzkin paths with a different
set of weights than the ones that come out of our interpretation. Thus
we can view $(e_{n,0}(2p,4q))_{n \geq 0}$ as a $(p,q)$-analogue of
this sequence.

\end{enumerate}
\end{rem}\ \\
{\bf Example 3.7.}
Wang and Zhu \cite{WZ16} observed that the sequence $(\overline{N}_{n}(q))_{n \geq 0}$ is a
Stieltjes moment sequence of polynomials where
$$\overline{N}_n(q) = \sum_{k=1}^n \frac{1}{n}\binom{n}{k-1}\binom{n}{k} q^k.$$
It follows from Theorem \ref{H} that the sequence $(\overline{N}_{n}(p,q))_{n \geq 0}$ is a
Stieltjes moment sequence of polynomials where
$$\overline{N}_n(p,q) = \sum_{k=1}^n \frac{1}{n}\binom{n}{k-1}\binom{n}{k} q^kp^{n-k}.$$

\begin{proposition} The sequence
$\left(\sum_{k=1}^n \frac{1}{n}\binom{n}{k-1}\binom{n}{k} p^{n-k}q^k\right)_{n \geq 0}$
is a Stieltjes moment sequence of polynomials.
\end{proposition}
Taking $a=q$, $b=pq$, $c=p+q$ and $e=pq$ in \eqref{J-pqst},
we obtain the generating function of the $(f_{n,0}(q))$ is
$$d_F(x,p,q)=\sum_{n \geq 0} f_{n,0}(p,q) x^n =  \frac{2}{1+(p-q)x+\sqrt{1-2(p+q)x+(p-q)^2x^2}}.$$

\begin{rem}
\begin{enumerate}
\item If we set $p=q=1$ in the sequence of $(f_{n,0}(p,q))_{n \geq 0}$, we
obtain the sequence of Catalan numbers $C_0,C_1,C_2, \ldots $.
Thus we can view $(f_{n,0}(p,q))_{n \geq 0}$ as a $(p,q)$-analogue of
the Catalan numbers.

 \item If we set $p=2$ and $q=1$ in the sequence of $(f_{n,0}(p,q))_{n \geq 0}$, we
obtain the sequence $1,1,3,11,45,19,903,4279, \ldots $ which is
the sequence of little Schr\"oder numbers \cite[A001003]{OEIS}.
Thus we can view $(f_{n,0}(2p,q))_{n \geq 0}$ as a $(p,q)$-analogue of
the little Schr\"oder numbers.

 \item If we set $p=1$ and $q=2$ in the sequence of $(f_{n,0}(p,q))_{n \geq 0}$, we
obtain the sequence $1,2,6,22,90,394,1806, \ldots $ which is
the sequence of large Schr\"oder numbers \cite[A0006318]{OEIS}.
Thus we can view $(f_{n,0}(p,2q))_{n \geq 0}$ as a $(p,q)$-analogue of
the large Schr\"oder numbers.

\item If we set $p=2$ and $q=2$ in the sequence of $(f_{n,0}(p,q))_{n \geq 0}$, we
obtain the sequence $1,2,8,40,224,1344,8448,54912,  \ldots $ which is
the sequence  \cite[A151374]{OEIS}. This sequence counts the number
of paths that start at (0,0) and stay in the first quadrant consisting of
$2n$ steps $(1,1)$, $(-1,-1)$, and $(-1,0)$.
\end{enumerate}
\end{rem}
It follows that for any $a,b \geq 0$,
the sequence $(f_{n,0}(a,b))_{n \geq 0}$ is
a Stieltjes moment sequence. In particular, the
Catalan numbers $C_n$, the little Schr\"oder numbers $S_n$, and the large Schr\"oder numbers $r_n$ are all Stieltjes moment sequences.\\
\ \\
{\bf Example 3.8.}
Wang and Zhu \cite{WZ16} proved that
$(S_n(q))_{n \geq  0}$ is a Stieltjes moment sequence of polynomials where
$S_n(q) = \sum_{k=1}^n S_{n,k} q^k$
and $S_{n,k}$ is Stirling number of the second kind which
counts the number of set partitions of $\{1,2, \ldots, n\}$ into
$k$ parts.  It follows form
Theorem \ref{H} that $(S_n(p,q))_{n \geq  0}$ is a Stieltjes moment sequence of polynomials where
$S_n(p,q) = \sum_{k=1}^n S_{n,k} q^kp^{n-k}$.

\begin{proposition} The sequence
$\left(\sum_{k=1}^n S_{n,k} p^{n-k}q^k\right)_{n \geq 0}$ is a
Stieltjes moment sequence of polynomials.
\end{proposition}

It is straightforward to obtain the generating functions of the
$g_{n,0}(p,q)$s.  That is,
Let $\mathbb{S}_k(x) = \sum_{n \geq k} S(n,k) x^n$. Then
it is well-known that $\mathbb{S}_0(x) =1$ and for $k \geq 1$,
$$\mathbb{S}_k(x) = \frac{x^k}{(1-x)(1-2x) \cdots (1-kx)}.$$
It follows that
$$\sum_{n \geq 0} x^n \sum_{k =0}^n S(n,k)q^k =
\sum_{k \geq 0} q^k \mathbb{S}_k(x) = 1 + \sum_{k \geq 1}
\frac{q^k x^k}{(1-x)(1-2x) \cdots (1-kx)}.$$
Replacing $x$ by $xp$ and $q$ by $q/p$ in the generating function
above gives
$$\sum_{n\geq 0}S_n(p,q)x^n=1+ \sum_{k\geq 1}\frac{q^kx^k}{(1-px)(1-2px)\cdots(1-kpx)}.$$

\begin{rem}
\begin{enumerate}
\item If we set $p=q=1$ in the sequence of $(g_{n,0}(p,q))_{n \geq 0}$, we
obtain the sequence of Bell numbers $B_0,B_1,B_2, \ldots $.
If we set $p=q=k$ in the sequence of $(g_{n,0}(p,q))_{n \geq 0}$, we
obtain the sequence $(k^nB_n)_{n \geq 0}$. In particular for all $k \geq 1$,
the sequence $(k^nB_n)_{n \geq 0}$ is a Steiltjes moment sequence.
Thus we can view $(g_{n,0}(p,q))_{n \geq 0}$ as a $p,q$-analogue of
the Bell numbers.

 \item If we set $p=2$ and $q=1$ in the sequence of $(g_{n,0}(p,q))_{n \geq 0}$, we
obtain the sequence $1,1,3,11,49,257,1539,  \ldots $ which is the sequence
\cite[A004211]{OEIS}.

 \item If we set $p=1$ and $q=2$ in the sequence of $(g_{n,0}(p,q))_{n \geq 0}$, we
obtain the sequence $1,2,6,22,94,454,2430, \ldots $ which is
the sequence \cite[A001861]{OEIS}.

\item If we set $p=2$ and $q=2$ in the sequence of $(g_{n,0}(p,q))_{n \geq 0}$, we
obtain the sequence $1,2,8,40,224,1344,8448,54912,  \ldots $ which is
the sequence  \cite[A055882]{OEIS}.
\end{enumerate}
\end{rem}
\ \\
{\bf Example 3.9.}
For any $\sg = \sg_1 \cdots \sg_n \in S_n$, let
$\mathrm{des}(\sg) = |\{i:\sg_i > \sg_{i+1}\}|$ and
$\mathrm{ris}(\sg) = |\{i:\sg_i < \sg_{i+1}\}|$.
Wang and Zhu \cite{WZ16} proved that $(E_{n+1}(q))_{n \geq 0}$ is a Stieltjes moment sequence
of polynomials where for $n \geq 1$,
$$E_n(q) = \sum_{k=0}^{n-1} E_{n,k}q^k = \sum_{ \sg \in S_n} q^{\mathrm{des}(\sg)}.$$
 It follows from Theorem \ref{H} that
$(E_{n+1}(p,q))_{n \geq 0}$ is a Stieltjes moment sequence
of polynomials where
$$E_n(p,q) = \sum_{k=0}^{n-1} E_{n,k}q^kp^{n-k} = \sum_{ \sg \in S_n} q^{\mathrm{des}(\sg)}p^{\mathrm{ris}(\sg)}.$$

\begin{proposition}
The sequence
$\left(\sum_{\sg \in S_n} p^{\mathrm{ris}(\sg)} q^{\mathrm{des}(\sg)}\right)_{n \geq 1}$ is a
Stieltjes moment sequence of polynomials.
\end{proposition}

It is easy to obtain a generating function for the $E_{n}(p,q)$s.
That is, it is well-known that
$$1 + \sum_{n\geq 1} \frac{x^n}{n!} \sum_{\sg \in S_n} q^{\mathrm{des}(\sg)}
= \frac{q-1}{q -e^{x(q-1)}}.$$
See \cite{EC2}.
Replacing $x$ by $px$ and $q$ by $q/p$ in this generating function
we obtain that
$$1 + \sum_{n\geq 1} \frac{x^n}{n!} \sum_{\sg \in S_n}
q^{\mathrm{des}(\sg)} p^{\mathrm{ris}(\sg)+1}
= \frac{q-p}{q -pe^{x(q-p)}}.$$
Thus
$$\sum_{n\geq 1} \frac{x^n}{n!} \sum_{\sg \in S_n}
q^{\mathrm{des}(\sg)} p^{\mathrm{ris}(\sg)} =\frac{e^{x(q-p)}-1}{q-pe^{x(q-p)}}.$$

\end{document}